\documentclass[12pt,reqno]{amsart}

\usepackage{amsfonts,amsmath,amssymb,amscd,amsthm}
\usepackage{graphicx}
\usepackage[all]{xy}
\usepackage{url}

\newtheorem{Lem}{Lemma}[section]
\newtheorem{Prop}[Lem]{Proposition}
\newtheorem{Cor}[Lem]{Corollary}
\newtheorem{Thm}[Lem]{Theorem}
\newtheorem{Def}[Lem]{Definition}
\newtheorem{Rem}[Lem]{Remark}
\newtheorem{Expl}[Lem]{Example}

\DeclareMathOperator{\Tor}{Tor}
\DeclareMathOperator{\Mon}{Mon}
\DeclareMathOperator{\supp}{supp}
\DeclareMathOperator{\set}{set}
\DeclareMathOperator{\cx}{cx}
\DeclareMathOperator{\depth}{depth}
\DeclareMathOperator{\m}{m}
\newcommand{\NN}{\mathbb{N}}
\newcommand{\ZZ}{\mathbb{Z}}
\newcommand{\RR}{\mathbb{R}}

\def\qed{\ifhmode\textqed\fi
	\ifmmode\ifinner\hfill\quad\qedsymbol\else\dispqed\fi\fi}
\def\textqed{\unskip\nobreak\penalty50
	\hskip2em\hbox{}\nobreak\hfill\qedsymbol
	\parfillskip=0pt \finalhyphendemerits=0}
\def\dispqed{\rlap{\qquad\qedsymbol}}

\begin{document}
\title{Mapping cones of monomial ideals over exterior algebras}
\author{Marilena Crupi, Antonino Ficarra, Ernesto Lax}

\address{Marilena Crupi, Department of mathematics and computer sciences, physics and earth sciences, University of Messina, Viale Ferdinando Stagno d'Alcontres 31, 98166 Messina, Italy}
\email{mcrupi@unime.it}

\address{Antonino Ficarra, Department of mathematics and computer sciences, physics and earth sciences, University of Messina, Viale Ferdinando Stagno d'Alcontres 31, 98166 Messina, Italy}
\email{antficarra@unime.it}

\address{Ernesto Lax, Department of mathematics and computer sciences, physics and earth sciences, University of Messina, Viale Ferdinando Stagno d'Alcontres 31, 98166 Messina, Italy}
\email{erlax@unime.it}


\begin{abstract} 
	Let $K$ be a field, $V$ a finite dimensional $K$-vector space and $E$ the exterior algebra of $V$. We analyze iterated mapping cone over $E$. If $I$ is a monomial ideal of $E$ with linear quotients, we show that the mapping cone construction yields a minimal graded free resolution $F$ of $I$ via the Cartan complex. Moreover, we provide an explicit description of the differentials in $F$ when the ideal $I$ has a regular decomposition function. Finally, we get a formula for the graded Betti numbers of a new class of monomial ideals including the class of strongly stable ideals.
\end{abstract}

\maketitle

%%%%%%%%%%%%%%%%%%%%%%
\section{Introduction}
%%%%%%%%%%%%%%%%%%%%%%
A popular topic in combinatorial commutative algebra is to detect minimal graded free resolutions of classes of monomial ideals (see, for instance, \cite{DM, IP} and the reference therein). In \cite{JTak}, Herzog and Takayama have shown that one can iteratively construct a free resolution of any monomial ideal in a polynomial ring by considering the mapping cone of the map between certain complexes, adding one generator at a time. Especially when the ideal has linear quotients, the complexes and the maps are well--behaved, and the procedure determines a minimal resolution \cite{JTak} by means of the Koszul complex. Some meaningful examples of free resolutions which arise as iterated mapping cones are the Eliahou--Kervaire resolution of stable ideals \cite{EK} (see, also, \cite{EC}) and the Taylor resolution. \\
Let $K$ be a field and let $E=K\left\langle e_1,\ldots , e_n\right\rangle$ be the exterior algebra of a $K$-vector space $V$ with basis $e_1,\ldots,e_n$. 
A fundamental tool in the study of minimal graded free resolutions of graded $E$-modules is played by the Cartan complex \cite{AHH}. It is an infinite complex which has the structure of a divided power algebra, and furthermore, as in the case of the ordinary Koszul complex over the polynomial ring, one may compute the Cartan homology of an
$E$-module using long exact sequences attached to the partial sequences $e_1,\ldots, e_i$.  \\
The aim of this article is to determine a minimal graded free resolution $F$ of a monomial ideal in $E$ with linear quotients by iterated mapping cones and, in particular, to describe explicitly the free $E$-modules which appear in $F$ and thus to get information on the graded Betti numbers of the given ideal. 
Our main tool is the Cartan complex. As in the polynomial case \cite{JTak}, we are able to provide an explicit description of the differentials in $F$ for a special class of ideals with linear quotients which satisfy an additional condition.  A significant example of monomial ideals of $E$ with linear quotients is the class of stable ideals whose minimal graded free resolution has been completely described by Aramova, Herzog and Hibi in \cite{AHH} (see, also, \cite{AAH1}). Recently, minimal graded free resolutions in the context of skew commutative algebras have been studied in \cite{FH,EEE22}. Moreover, new meaningful results related to iterated mapping cone
procedure have been recently carried out in \cite{KV}, where the author has replaced the
polynomial ring by any commutative strongly Koszul algebra.\\
The plan of the article is as follows. Section \ref{sec:2}, contains some notions that will be used throughout the article. 
In Section \ref{sec3}, we discuss the mapping cone over the exterior algebra $E$. We prove that if $I$ is a monomial ideal with linear quotients, the mapping cone construction determines a minimal graded free resolution of $I$,  and thus we describe explicitly a basis for each free module in the resolution (Theorem \ref{prop:iterated}). A crucial fact is the exterior version of the ``Rearrangement lemma" of Bjorner and Wachs \cite{BW}. Hence, if $G(I)$ is the minimal set of monomial generators of $I$, analyzing some special sets of positive integers that can be associated to $G(I)$, we get a formula for the graded Betti numbers of $I$ (Corollary \ref{Cor:formulaBetti}) via the notion of weak composition (Definition \ref{def:weak}). Then, we obtain  some formulas for the graded Poincar\'e series, the complexity and the depth of $I$ (Corollaries \ref{cor:poinc}, \ref{cor:compldepth}). In Section \ref{sec4}, we introduce the notion of regular decomposition function (Definition \ref{def:regfunc}) of an ideal of $E$ with linear quotients. Our main result is Theorem \ref{thm:main1} that gives an explicit description of the differentials of the resolutions of all ideals of $E$ with linear quotients which admit a regular decomposition function. Finally, in Section \ref{sec5}, firstly, we introduce the class of  ${\bf t}$-spread strongly stable ideals in $E$ (Definition \ref{def:stronspread}), then we discuss their main properties and, as a consequence, we compute their graded Betti numbers (Corollary \ref{cor:formula}) via the results in Section \ref{sec3}. Such a formula generalizes the one stated in \cite{AHH} for the ordinary strongly stable ideal in $E$.\\
All examples in this article have been verified computationally using the software Macaulay2 \cite{GDS} and the package \cite{AC}.

%%%%%%%%%%%%%%%%%%%%%%%%%%%%%%%%%%%%
\section{Preliminaries and notation}\label{sec:2}
%%%%%%%%%%%%%%%%%%%%%%%%%%%%%%%%%%%%
Let $K$ be a field. We denote by $E=K\left\langle e_1,\ldots , e_n\right\rangle$ the exterior algebra
of a $K$-vector space $V$ with basis $e_1,\ldots,e_n$.

As a $K$-algebra $E$ is standard graded with defining relations $v\wedge v=0$, for $v\in E$ and $e_i\wedge e_j= - e_j\wedge e_i$. 
For any subset $\mu=\{i_1,\ldots,i_d\}$ of $\{1,\dots,n\}$ with $i_1<i_2< \cdots < i_d$ we write $e_\mu=e_{i_1} \wedge \ldots \wedge e_{i_d}$, and call $e_{\mu}$ a \emph{monomial} of degree $d$.  We set $e_{\mu}=1$, if $\mu = \emptyset$. Moreover, setting $\sigma(\tau, \mu)=\vert \{(i, j): i \in \tau, j\in \mu, i>j\}\vert$, for any $\tau,\mu \subseteq \{1,\dots,n\}$, we have $e_\tau \wedge e_\mu = (-1)^{\sigma(\tau, \mu)}e_{\tau\cup \mu}$, if $\tau\cap \mu=\emptyset$ and $e_\tau \wedge e_\mu=0$, otherwise. The set of monomials in $E$ forms a $K$-basis of $E$ of cardinality $2^n$. An element $f \in E$ is called \textit{homogeneous} of degree $j$ if $f \in E_j$, where $E_j=\bigwedge^j V$. From the previous relations, one has that $f\wedge g= (-1)^{\deg(f)\deg(g)}g\wedge f$ for any two homogeneous elements $f$ and $g$ in $E$ (see, for instance, \cite{JT}). In order to simplify the notation, we put $fg=f \wedge g$ for any two elements $f$ and $g$ in $E$.

We denote by $\mathcal{G}$ the category of graded $E$-modules whose morphisms are the homogeneous $E$-module homomorphisms of degree $0$ \cite{JT}.
Every $E$-module $M \in \mathcal{G}$ has a minimal graded free resolution $F$ over $E$:
\[
F: \ldots \to F_2 \stackrel{\,\,d_2 \,\,}{\rightarrow} F_1 \stackrel{\,\, d_1 \,\,}{\rightarrow} F_0 \to M \to 0,
\]
where $F_i=\oplus_jE(-j)^{\beta_{i,j}(M)}$. The integers $\beta_{i,j}(M)=\dim_K\Tor_{i}^{E}(M,K)_j$ are called the \textit{graded Betti numbers} of $M$, whereas the numbers $\beta_{i}(M)=\sum_j\beta_{i,j}(M)$ are called the \textit{ith total Betti numbers} of $M$.
Moreover, the \emph{graded Poincar\'e series} of $M$ over $E$ is defined as $\mathcal{P}_M(s, t)= \sum_{i,j\ge 0}\beta_{i, j}(M)t^is^j$ \cite{AHH}.\\

A graded ideal $I$ of $E$ is a monomial ideal if $I$ is generated by monomials.

Let $e_\mu=e_{i_1} \cdots e_{i_d}\neq 1$ be a monomial in $E$. We define
\[
\supp(e_\mu)= \{\mbox{$i$\,:\,$e_i$ divides $e_\mu$}\} = \mu
\]
and we write
\[
\m(e_\mu)=\max\{i: i \in \supp(e_\mu)\}= \max\{i: i \in \mu\}.
\]
We set $\m(e_\mu)=0$, if $e_\mu=1$.

\begin{Def}\em
	Let $I$  be a monomial ideal of $E$. $I$ is called \textit{stable} if for each monomial $e_{\mu}\in I$ and each $j < \m(e_{\mu})$ one has $e_j e_{{\mu} \setminus \{\m(e_{\mu})\}} \in I$.
	$I$ is called \textit{strongly stable} if for each monomial $e_{\mu} \in I$ and each $j \in \mu$ one has $e_ie_{\mu \setminus \{j\}} \in I$ for all $i<j$.
\end{Def}

If $I$  is a monomial ideal in $E$, we denote by $G(I)$ the unique minimal set of monomial generators of $I$, and by  $G(I)_d$ the set of all monomials $u\in G(I)$ such that $\deg(u)=d$, $d>0$. 
One can verify that the defining property of a strongly stable ideal only needs to be checked on the generators (see, for instance \cite[Remark 2.2]{AC}).\\

We close this section with some comments. As in the polynomial ring, every element of $E$ of the type $ce_\mu$, with $c\in K$ and $e_\mu$ a monomial, is called a \emph{term}. In what follows, to avoid ambiguity and with the aim of simplifying the computations, we refer to any term of the form $ce_{\mu}$ with $c\in\{-1, 1\}$,  as a monomial. In other words we generalize the classical notion of monomial recalled at the beginning of the section. This choice does not affect the classical development of the exterior algebra theory.

\subsection{A glimpse to the Cartan Complex}\label{sub:Cartan}
The Cartan complex for the exterior algebra $E=K\left\langle e_1,\ldots , e_n\right\rangle$ plays the role of the Koszul complex for the symmetric algebra.
We refer to \cite{AHH, JT} for more details on this subject.

Let $\NN$ be the set of all non negative integers and for a positive integer $n$, let \linebreak$[n] = \{1, \ldots, n\}$. Moreover, for a $n$-tuple $a= (a_1, \ldots, a_n)\in \NN^n$, let $\vert a\vert = a_1+ \cdots + a_n$.

Let $v = v_1,\ldots, v_m$ be a sequence of linear forms in $E_1$. The Cartan complex $C.(v;E)$ of the sequence $v$ with values in $E$ is defined as the complex
whose $i$-chains $C_i(v;E)$ are the elements of degree $i$ of the free divided power algebra $E\langle x_1, \ldots, x_m\rangle$, where $E\langle x_1, \ldots, x_m\rangle$ is the polynomial ring over $E$ in the set of variables
\[x_i^{(j)}, \quad i=1, \ldots, m,\quad j=1,2, \ldots\]
modulo the relations
\[x_i^{(j)}x_i^{(k)} = \binom{j+k}jx_i^{(j+k)}.\]
We set $x_i^{(0)}=1$, $x_i^{(1)}=x_i$ for $i=1, \ldots, m$ and $x_i^{(a)}=0$ for $a<0$. The algebra $E\langle x_1, \ldots, x_m\rangle$ is a free $E$-module with basis
\[x^{(a)} = x_1^{(a_1)}x_2^{(a_2)}\cdots x_m^{(a_m)}, \quad a = (a_1, \ldots, a_m)\in \NN^m.\]
We say that $x^{(a)}$ has degree $i$ if $|a| = i$.
If $x^{(a)}\neq 1$, we define $\supp(x^{(a)}) = \{i\in [m] : a_i\neq 0\}$ and set $\supp(x^{(a)})=\emptyset$, for 
$x^{(a)}=1$.\\

One has $C_i(v;E)= \oplus_{|a|=i} Ex^{(a)}$. The $E$-linear differential $\partial$ on $C.(v;E)$ is defined as follows: for $x^{(a)} =x_1^{(a_1)}x_2^{(a_2)}\cdots x_m^{(a_m)}$, we set
\[\partial(x^{(a)}) = \sum_{i=1}^mv_i x_1^{(a_1)}\cdots x_i^{(a_i-1)}\cdots x_m^{(a_m)}.\]

If $\mathcal{G}$ is the category of graded $E$-modules above defined and $M\in \mathcal{G}$, one can define the complex
$C.(v;M) = M\otimes_E C.(v;E)$ and set $H_i(v;M) = H_i(C.(v;M))$. We call $H_i(v;M)$ the $i$th Cartan homology
module of $v$ with respect to $M$. One can observe that each $H_i(v;M)$ is a graded $E$-module. In \cite[Theorem 2.2]{AHH}, the authors showed that the Cartan complex $C.(v;E)$ is a minimal free resolution of $E/(v)$ and, as a consequence, they proved that 
for any graded $E$-module $M$ and each $i\ge 0$, there is the natural isomorphism $\Tor_i^E(M,K)\cong H_i(e_1,\ldots, e_n;M)$ of graded $E$-modules. Such a result allows to compute the graded Betti numbers of $M$.

The next important formula for determining the graded Betti numbers of a stable ideal $I$ of $E$ has been stated by Aramova, Herzog and Hibi via the Cartan complex \cite[Corollary 3.3]{AHH}:
\begin{equation}\label{eq:Bettistable}
\beta_{i, i+j}(I) = \sum_{u\in G(I)_j}\binom{\m(u)+i-1}{\m(u)-1}, \quad \mbox{for all $i\ge 0$}.
\end{equation}
Indeed, for every $i>0$, a basis of the modules $H_i(e_1,\ldots, e_n;E/I)$ is given by the homology classes of the cycles 
\[u_{\m(u)}x^{(a)}, \,\, u\in G(I),\,\, \vert a \vert =i,\,\,  \max(a) = \m(u),\] 
where $\max(a) = \max\supp(x^{(a)})= \max\{i\in [n]: a_i\neq 0\}$ ($a\in \NN^n$) and $u_{\m(u)}=e_{\mu\setminus\{\m(u)\}}$, for $u=e_\mu$.

%%%%%%%%%%%%%%%%%%%%%%%%%%%%%%%%%%%%%%%%%%%%%%%%%%%%%%
\subsection{A glimpse to ideals with linear quotients} \label{subsec2}
%%%%%%%%%%%%%%%%%%%%%%%%%%%%%%%%%%%%%%%%%%%%%%%%%%%%%%
Monomial ideals with linear quotients have been introduced by Herzog and Takayama
\cite{JTak} to study resolutions that arise as iterated mapping cones. 

\begin{Def}\em
	A  monomial ideal $I$ of $E$ is said to have \textit{linear quotients} if for some order $u_1,\ldots, u_r$ of the elements of $G(I)$ the ideals $0 : (u_1)$ and $(u_1, \ldots, u_{i-1}):(u_i)$, for $i = 2, \ldots, r$, are generated by a subset of the set $\{e_1, \ldots, e_n\}$.
\end{Def}

\begin{Rem}\em
	Differently from the definition of linear quotients over the polynomial ring, we need the condition $0 : (u_1)$ has to be generated by a subset of $\{e_1, \ldots, e_n\}$. This is always the case, since $0 : (u_1)= (\supp(u_1))$. 
\end{Rem}

The ideal $I = (e_2, e_3e_4)$ of $E = K\langle e_1, \ldots, e_4\rangle$ has linear quotients with respect to this order of the generators. In fact,
$0 :(e_2) = (e_2)$, and $(e_2):(e_3e_4) = (e_2, e_3, e_4)$.
But $I$ does not satisfy the conditions for linear quotients, when employing the reversed order on the generators. Indeed, for the order $e_3e_4,e_2$, we have $0:(e_3e_4)=(e_3, e_4)$ and $(e_3e_4):(e_2) = (e_2, e_3e_4)$.\\

Following \cite{JTak}, for a monomial ideal $I$ of $E$ with linear quotients with respect to some order of the elements $u_1, \ldots, u_r$ of $G(I)$, we define
\[
\set(u_j) =\{k\in [n] : e_k\in (u_1, \ldots, u_{j-1}):(u_j)\},\quad \mbox{for $j=2, \ldots, r$}.
\]
Moreover, we set
\[
\set(u_1)= \{k\in [n] : e_k\in 0:(u_1)\}= \supp\, u_1.
\]

One can observe that in contrast to the polynomial ring context, $\set(u)$ contains $\supp(u)$ for every $u\in G(I)$.

\begin{Rem}\em \label{ex:stable}
	A stable ideal $I$ of $E$ has linear quotients with respect to the reverse degree lexicographical order on the generators. Moreover, $\set(u) = \{i : i \le \m(u)\}$, \emph{i.e.}, $\set(u) = [\m(u)]$, for every $u\in G(I)$.
	
	For instance, let us consider the stable ideal  $I= (e_1e_2, e_1e_3, e_2e_3, e_3e_4e_5)$ of $E = K\langle e_1, \ldots, e_5\rangle$. Then $I$ has linear quotients with respect to the reverse lexicographic order $e_1e_2,e_1e_3,e_2e_3,e_3e_4e_5$. Indeed, $0:(e_1e_2) = (e_1, e_2)$, $(e_1e_2):(e_1e_3) = (e_1, e_2, e_3)$, $(e_1e_2, e_1e_3):(e_2e_3)= (e_1, e_2, e_3)$ and $(e_1e_2, e_1e_3, e_2e_3):(e_3e_4e_5)= (e_1, e_2, e_3, e_4,e_5)$. Therefore, $\set(e_1e_2) = \{1,2\}$, $\set(e_1e_3) = \{1,2,3\}$, $\set(e_2e_3) = \{1,2, 3\}$, $\set(e_3e_4e_5) = \{1,2, 3, 4, 5\}$, and so
	$\vert \set(u)\vert = \m(u)$, for all $u\in G(I)$.
\end{Rem}

Finally, we have the following hierarchy of monomial ideals in $E$:
\[
\mbox{strongly stable ideals \,\, $\Rightarrow$\,\, stable ideals \,\, $\Rightarrow$\,\, ideals with linear quotients}
\]

Let $I$ be a monomial ideal of $E$ with linear quotients with respect to a minimal set of monomial generators $\{u_1, \ldots, u_r\}$. We observe that not necessarily $\deg u_1 \le \cdots \le \deg u_r$. Let us consider the ideal $I = (e_1e_2, e_2e_3e_4, e_1e_3)$ of $E=K\langle e_1, \ldots, e_4\rangle$. $I$ has linear quotients for the given order of the generators. In fact , $0:(e_1e_2) = (e_1, e_2)$, $(e_1e_2):(e_2e_3e_4) = (e_1, e_2, e_3, e_4)$ and $(e_1e_2, e_2e_3e_4):(e_1e_3)= (e_1, e_2, e_3)$.\\

Now, let $I\subset E$ be a monomial ideal with linear quotients with respect to homogeneous generators $u_1, \ldots, u_r$. We say that the order $u_1, \ldots, u_r$ is \emph{degree increasing} if $\deg(u_1)\le \dots \le\deg(u_r)$. By using exterior algebra methods \cite[Lemma 4.2.1]{Tthesis}, one can prove that if $I\subset E$ is a monomial ideal with linear quotients then $I$ has also linear quotients with respect to a degree increasing order. Such a statement can be seen as the exterior version of the ``Rearrangement lemma" of Bjorner and Wachs \cite{BW}. 

Hence, from now on when we say that a monomial ideal $I\subset E$ has linear quotients with respect to some order of $G(I)= \{u_1, \ldots,u_r\}$, we assume tacitly that such an order is a degree increasing order, \emph{i.e.}, $\deg(u_1)\le\dots\le\deg(u_r)$.\\

We close the section with some comments on resolutions of monomial ideals with linear quotients.

It is well--known that in the polynomial ring case, monomial ideals generated in one degree $d$ with linear quotients have $d$-linear resolutions (see, for instance \cite[Proposition 8.2.1.]{JT}). The same property holds for monomial ideals in an exterior algebra (see, \cite[Corollary 5.4.4.]{GK}). 

The property of having linear quotients can be related to \emph{componentwise linearity}. This concept has been introduced for ideals in a polynomial ring by
Herzog and Hibi in \cite{HH1999} in order to generalize the Eagon and Reiner’s result that
the Stanley--Reisner ideal of a simplicial complex has a linear resolution if and only if the Alexander Dual of the complex is Cohen-Macaulay. Such a notion can be also introduced in the exterior algebra context (see, for instance, \cite[Definition 5.3.1]{GK}). More precisely, a module $M \in \mathcal{M}$ is called \emph{componentwise linear} if the submodule $M_{\langle j\rangle}$ of  $M$, which is generated by all homogeneous elements of degree $j$ belonging to $M$,
has a $j$--linear resolution for all $j\in \ZZ$. Moreover, as in the commutative case, one can prove that if $I$ is a monomial ideal of $E$ with linear quotients, then $I$ is a
componentwise linear ideal \cite[Theorem 5.4.5.]{GK}.

Finally, it is worth mentioning a result due to  Aramova,  Avramov and Herzog \cite[Corollary 2.2.]{AAH1}, which states that if $I$ is an ideal generated by squarefree monomials in
a polynomial ring $S=K[x_1, \ldots, x_n]$, and $J$ denotes the corresponding monomial ideal in $E$, the ideal $I$ has a linear free resolution over $S$ if and only if the ideal $J$ has a linear free resolution over $E$.

%%%%%%%%%%%%%%%%%%%%%%%%%%%%%%%%%%%%%%
\section{Resolutions by mapping cones} \label{sec3}
%%%%%%%%%%%%%%%%%%%%%%%%%%%%%%%%%%%%%%
In this section, we determine a minimal graded free resolution of monomial ideals with linear quotients by \emph{mapping cones}. 

Roughly speaking,  if $A$ and $B$ are two chain complexes of graded $E$-modules of the category $\mathcal{G}$ and $\psi: A\rightarrow B$ is a complex homomorphism.
The mapping cone of $\psi$ is a new chain complex denoted by $C(\psi)$ with $C(\psi)_i = A_{i-1}\oplus B_i$ for all $i$, and the differentials related to the ones in $A$ and $B$ by a special formula.
We refer to \cite{DE, CW} for more details on this topic. In particular, the mapping cone of a chain map between two free resolutions gives again a free resolution which is not in general minimal. 

We apply this concept in a special situation. Let $I$ be a monomial ideal of $E$ and assume that $I$ has linear quotients with respect the order $u_1, \ldots, u_r$ of its generators. The basic idea is to exploit the short exact sequences that arise from adding one generator of $I$ at a time, and to iteratively construct the resolution as a mapping cone of an appropriate map between previously constructed complexes. 

For this aim, set $I_j = (u_1, \ldots, u_j)$ and $L_j = (u_1, \ldots, u_j):(u_{j+1})$, $j=1, \ldots, r$. Since, $I_{j+1}/I_j\simeq E/L_j$, then 
we get the following exact sequences of $E$-modules
\[0 \longrightarrow E/L_j \xrightarrow{\ \alpha}  E/I_j \xrightarrow{\ \beta} E/I_{j+1}\longrightarrow 0,\]
where $\alpha:E/L_j \longrightarrow E/I_j$ is multiplication by $u_{j+1}$. Consider the ideal $L_j$. By definition $\set(u_{j+1}) =\{k\in [n] : e_k\in (u_1, \ldots, u_j):(u_{j+1})=L_j\}$, for $j=1, \ldots, r$.
Thus, since $I$ has linear quotients with respect to the order $u_1, \ldots, u_r$ of its generators, we have that $L_j=(e_{k_1}, \ldots, e_{k_\ell})$, with $\{k_1, \ldots, k_\ell\}=\set(u_{j+1})$, for $j=1, \ldots, r$.
Hence, if $C^{(j)}= C.(e_{k_1}, \ldots, e_{k_\ell}; E)$ is the Cartan complex for the sequence $e_{k_1}, \ldots, e_{k_\ell}$, we have that $C^{(j)}$ is a minimal graded free resolution of $E/L_j$ \cite{AHH}.

Let $F^{(j)}$ be a minimal graded free resolution of $E/I_j$. Since the modules in $C^{(j)}$ are free, thus projective, there exists a complex homomorphism $\psi^{(j)} :C^{(j)}\rightarrow F^{(j)}$, that is a sequence of maps $\psi_i^{(j)}: C^{(j)}_i\rightarrow F^{(j)}_i$ $(i\ge0)$, called the \textit{comparison maps}, which \textit{lifts} the map $\psi_{-1}^{(j)}=\alpha$ and makes the following diagram

\begin{equation*}\label{diagram:ABpsi-1}
\begin{gathered}
\xymatrixcolsep{2.5pc}\xymatrix{
	\displaystyle
	C^{(j)}:\cdots\ar[r] &C^{(j)}_2 \ar[d]_{\psi^{(j)}_2}\ar[r]^{{d_2^{C^{(j)}}}} & C^{(j)}_1 \ar[d]_{\psi^{(j)}_1}\ar[r]^{{d_1^{C^{(j)}}}} & C^{(j)}_0\ar[d]_{\psi^{(j)}_0}\ar[r]^{{d_0^{C^{(j)}}}}&E/L_j\ar[d]^{\psi_{-1}^{(j)}=\alpha}\ar[r]&0\\
	F^{(j)}:\cdots\ar[r] &F^{(j)}_2 \ar[r]_{d_2^{F^{(j)}}} & F^{(j)}_1 \ar[r]_{d_1^{F^{(j)}}} & F^{(j)}_0\ar[r]_{d_0^{F^{(j)}}}&E/I_j\ar[r]&0
}
\end{gathered}
\end{equation*}
commutative. The homomorphism $\psi^{(j)}$ gives rise to an acyclic complex $C(\psi^{(j)})$, called the \emph{cone complex} of $\psi^{(j)}$, whose $0$th homology module is $H_0(C(\psi^{(j)}))=\textup{coker}(\alpha)=(E/I_j)/\textup{Im}(\alpha)\cong(E/I_j)/(E/L_j)\cong E/I_{j+1}$, that is,  $C(\psi^{(j)})$ is a free resolution of $E/I_{j+1}$ (see, for instance, \cite[Appendix A3.12]{DE}).
Thus by iterated mapping cones one obtains step by step a graded free resolution of $E/I$. 

More in detail, the cone complex $C(\psi^{(j)})$ is defined as follows:
\begin{enumerate}
	\item[(i)] let $C(\psi^{(j)})_0=C^{(j)}_0$, and $C(\psi^{(j)})_i=C^{(j)}_{i-1}\oplus F^{(j)}_{i}$, for $i>0$;
	\item[(ii)] let $d_0=\varphi\circ d_0^{F^{(j)}}$, $d_1=(0, \psi_{0}+d_1^{F^{(j)}})$, and $d_i=(-d_{i-1}^{C^{(j)}}, \psi_{i-1}+d_i^{F^{(j)}})$, for $i>1$. 
\end{enumerate}
This procedure, known as the \textit{mapping cone}, may be visualized as follows:

\[
\xymatrixcolsep{2.5pc}\xymatrix{
	\displaystyle
	C^{(j)}[-1]:\cdots\ar[r] &C^{(j)}_2 \ar @{} [d] |{\text{\large$\oplus$}} \ar[rd]^{\psi^{(j)}_2}\ar[r]^{{d_2^{C^{(j)}}}} & A_1 \ar @{} [d] |{\text{\large$\oplus$}}  \ar[rd]^{\psi_1^{(j)}}\ar[r]^{{d_1^{C^{(j)}}}} & A_0 \ar @{} [d] |{\text{\large$\oplus$}} \ar[rd]^{\psi^{(j)}_0}\ar[r]^-{{d_0^{C^{(j)}}}}& E/L_j\ar[r]&0&\\
	F^{(j)}:\cdots\ar[r]&F^{(j)}_3\ar[r]_{{d_3^{F^{(j)}}}} &F^{(j)}_2 \ar[r]_{{d_2^{F^{(j)}}}} & F^{(j)}_1 \ar[r]_{{d_1^{F^{(j)}}}}  &F^{(j)}_0\ar[r]_{{d_0^{F^{(j)}}}} &E/I_j\ar[r]&0
}
\]
Here $C^{(j)}[-1]$ is the complex $C^{(j)}$ \textit{homologically shifted} by $-1$.

Our aim is to prove that the free resolution $C(\psi^{(j)})$ is minimal. 

The next statement follows the arguments of the corresponding result in the polynomial ring context \cite[Lemma 1.5]{JTak}. We elaborate on
this argument and stress the needed changes due to the different context.

For a non zero $n$-tuple $a= (a_1, \ldots, a_n)\in \NN^n$, let $\supp(a) = \{j\in [n]: a_j\neq 0\}$.

\begin{Thm}\label{prop:iterated}
	Let $I$ be a monomial ideal of $E$ with linear quotients with respect to the order $u_1,\dots,u_r$ of the minimal generators of $I$ and let $F$ be the iterated mapping cone derived from the sequence $u_1, \ldots, u_r$. Then $F$ is a minimal graded free resolution of $E/I$ and the free $E$-module $F_i$ which appears in the $i$th step of $F$ has homogeneous basis
	\[\{f(a; u), \,\, \mbox{with $u\in G(I)$, $a \in \NN^n$, $\supp(a) \subset \set(u)$, $|a| = i-1$}\},\] 
	with $\deg f(a; u) = |a| + \deg u$, for all $i>0$.
\end{Thm}
\begin{proof}
	Let $F^{(j)}$, $j\ge 1$, be a minimal graded free resolution of $E/I_j$. We prove by induction on $j$, that the set 
	$\{f(a; u), \mbox{with $u\in G(I_j)$, $a \in \NN^n$, $\supp(a) \subset \set(u)$, $|a| = i-1$}\}$ where $\deg f(a; u) = |a| + \deg u$, 
	forms a homogeneous basis of each $E$-module $F_i^{(j)}$, $i\ge 1$.
	
	For $j = 1$, the assertion is clear. 
	
	Let $j>1$. Setting $L_j=(e_{k_1}, \ldots, e_{k_\ell})$, with $\{k_1, \ldots, k_\ell\}=\set(u_{j+1})$, for $j=2, \ldots, r$, let 
	$C^{(j)}= C.(e_{k_1}, \ldots, e_{k_\ell}; E)$ be the Cartan complex for the sequence $e_{k_1}, \ldots, e_{k_\ell}$.  As we have previously pointed out, $C^{(j)}$ is  a minimal graded free resolution of $E/L_j$. 
	
	With the same notations as in Subsection \ref{sub:Cartan}, in homological degree $i-1$ the Cartan complex $C^{(j)}$ has the $E$-basis $x^{(a)}= x_{k_1}^{(a_{k_1})}x_{k_2}^{(a_{k_2})}\cdots x_{k_\ell}^{(a_{k_\ell})}$, $a = (a_{k_1}, \ldots, a_{k_\ell})\in \NN^\ell$, with $|a| = \sum_{q=1}^\ell a_{k_q}= i-1$, $\supp(x^{(a)})=\supp(a)\subset \set(u_{j+1})$.
	
	Since $F_i^{(j+1)} = C_{i-1}^{(j)}\oplus F_i^{(j)}$, we obtain the basis we are looking for from the inductive hypothesis by identifying the elements $x^{(a)}$ with $f(a; u_{j+1})$.

	In order to prove that $F^{(j+1)}$ is a minimal free resolution, it is sufficient to verify that Im$(\psi^{(j)})\subset (e_1, \ldots, e_n)F^{(j)}$. 
	
	Let $f(a; u_{j+1}) \in C_{i-1}^{(j)}$ and $\psi^{(j)}(f(a; u_{j+1}))= \sum_{i=1}^j\sum_bc_{b, i}f(b;u_i)$. Since we are only considering degree increasing orders, then $\deg u_{j+1}\ge \deg u_i$, for all $i=1, \ldots, j$ and, furthermore, $|b| = |a|-1$. Thus $\deg f(a; u_{j+1}) = |a| + \deg u_{j+1}>|b|+\deg u_i=\deg f(b;u_i)$, for all $b$ and $i$. Hence, $\deg c_{b, i}>0$ for all $b$ and $i$. The assertion follows.
\end{proof}

As a consequence of Theorem \ref{prop:iterated}, we obtain a formula for computing the graded Betti numbers of a monomial ideal with linear quotients. Next notion will be crucial for our aim. 

\begin{Def}\em \label{def:weak}
	A sequence $(a_1, a_2,\ldots, a_k)$ of integers fulfilling $a_i\ge 0$ for all $i$, and $a_1 + a_2 +\cdots + a_k= n$ is called a \textit{weak composition} of $n$. 
\end{Def}

The number of weak compositions of an integer $n$ in $k$ parts is given by $\binom{n+k-1}{k-1}= \binom{n+k-1}{n}$.\newpage

\begin{Cor}\label{Cor:formulaBetti}
	Let $I$ be a monomial ideal of $E$ with linear quotients. Then
	\[\beta_{i,i+j}(I)=\sum_{u\in G(I)_j}\binom{i+\vert \set(u)\vert -1}{\vert \set(u)\vert-1},\]
	for all $i, j$. Hence,
	\[\beta_i(I) = \sum_{u\in G(I)}\binom{i+\vert \set(u)\vert-1}{\vert \set(u)\vert -1},\]
	for all $i\ge 0$.
\end{Cor}
\begin{proof}
	From Theorem \ref{prop:iterated}, 
	\[\beta_{i,i+j}(I) = \vert\{a=(a_1, \ldots, a_n)\in \NN^n : \supp(a)\subset \set (u), \mbox{$u\in G(I)_j$ and $|a|=i$}\}\vert.\]
	On the other hand, one can observe that 
	\[\vert\{a=(a_1, \ldots, a_n)\in \NN^n : \supp(a)\subset \set (u), \mbox{$|a|=i$}\}\vert\]
	is the number of the weak compositions of the positive integer $i$ in $\vert \set(u)\vert$ parts. The assertions follow. 
\end{proof}

\begin{Expl} \em \label{ex:one}
	Let $I= (e_1e_3, e_1e_4, e_2e_4e_6)$ be a monomial ideal of $E = K\langle e_1, \ldots, e_6\rangle$. $I$ has linear quotients with respect to the order $e_1e_3,e_1e_4,e_2e_4e_6$. In fact, $0:(e_1e_3) = (e_1, e_3)$, $(e_1e_3):(e_1e_4) = (e_1, e_3, e_4)$, $(e_1e_3, e_1e_4):(e_2e_4e_6)= (e_1, e_2, e_4, e_6).$
	Therefore, setting $u_1=e_1e_3$, $u_2= e_1e_4$, $u_3=e_2e_4e_6$, one has
	\[\set(u_1) = \{1, 3\}, \,\,\, \set(u_2) = \{1, 3, 4\},\,\,\, \set(u_3)= \{1, 2, 4, 6\}.\]
	Clearly $\beta_0(I)=3$. In fact, $\binom{0+\vert \set(u)\vert-1}{\vert \set(u)\vert-1}=1$, for all $u\in G(I)$. Moreover,
	
	\begin{eqnarray*}
	\beta_1(I) &=& \vert\{a=(a_1, \ldots, a_6)\in \NN^6 : \supp(a)\subset \set (u), \mbox{$u\in G(I)$ and $\vert a \vert=1$}\}\vert \\
	&=& \sum_{i=1}^3\binom{1+\vert\set(u_i)\vert-1}{\vert\set(u_i)\vert-1} = 
	\binom 21+\binom 32 + \binom 43= 9.
	\end{eqnarray*}
	Let us compute $\beta_2(I)$. It is
	\begin{eqnarray*}
	\beta_2(I) &=& \vert\{a=(a_1, \ldots, a_6)\in \NN^6 : \supp(a)\subset \set (u), \mbox{$u\in G(I)$ and $\vert a \vert=2$}\}\vert\\
	&=& \sum_{i=1}^3\binom{2+\vert\set(u_i)\vert-1}{\vert\set(u_i)\vert-1} = \binom 31+\binom 42 + \binom 53= 19. 
	\end{eqnarray*}
	And so on. The Betti table of $I$ displayed by \emph{Macaulay2} \cite{GDS} is the following one:
	\[\begin{matrix}
	& 0 & 1 & 2& 3&4&5&6 & \cdots\\
	\text{total:}
	& 3 & 9 & 19 & 34 & 55 & 83 & 119 & \cdots\\
	2: & 2 & 5 & 9 & 14 & 20 & 27 & 35 & \cdots\\
	3: & 1 & 4& 10 & 20 & 35 & 56 & 84 & \cdots
	\end{matrix}\]\\
\end{Expl}

Corollary \ref{Cor:formulaBetti} yields the next results.

\begin{Cor}\label{cor:poinc}
	Let $I$ be a monomial ideal of $E$ with linear quotients. Then the graded Poincar\'e series of $I$ over $E$ is \[\mathcal{P}_{I}(s, t)= \sum_{i\ge 0} \sum_{j\ge 0} \sum_{u\in G(I)_j}(1+t)^{i+\vert \set(u)\vert -1}s^j.\]
\end{Cor}
\begin{proof}
	Since $\mathcal{P}_{I}(s, t)= \sum_{i,j\ge 0}\beta_{i, j}(I)t^is^j$, the desired formula follows from Corollary \ref{Cor:formulaBetti}. 
\end{proof}

Recall that if $\mathcal{M}$ is the category of finitely generated $\ZZ$-graded left and right $E$-modules $M$ satisfying $am=(-1)^{\deg a\deg m}ma$ for all homogeneous elements $a\in E$ and $m\in M$ the \emph{complexity} of $M$, which measures the growth rate of the Betti numbers of $M$, is defined as follows \cite{AAH1} (see, also, \cite{GK}):

\[\cx_E(M) = \inf\{c\in \ZZ: \beta_i(M)\le \alpha i^{c-1} \,\, \mbox{for some $\alpha\in \RR$ and for all $i\ge1$}\}.\]

In \cite{AAH1}, the authors introduced an important invariant of an $E$-module, the depth, in a similar way as the depth of a module over a polynomial ring. More in detail, for $M\in \mathcal{M}$, a linear form $v\in E_1$ is called $M$-regular if $0 : _M (v)  = vM$. A sequence of linear forms $v_1, \ldots,v_s$ is called an $M$-regular sequence if $v_1$ is $M$-regular, $v_i$ is $M/(v_1, \ldots, v_{i-1})M$-regular for $i = 2,\ldots, s$ and $M/(v_1, \ldots, v_s)M\neq 0$ \cite{AAH1, HM}. It is shown in \cite{AAH1} that all maximal $M$-regular sequences have the same length. This length is called the \emph{depth} of $M$ over $E$ and is denoted by $\depth_E(M)$. Such an invariant is closely related to the complexity at least if the ground field is infinite. Indeed, in such case, Aramova, Avramov and Herzog \cite[Theorem 3.2]{AAH1} proved that, for $M\in\mathcal{M}$
\begin{equation}\label{eq:AAH}
n= \cx_E(M)+\depth_E(M).
\end{equation}

\begin{Cor}\label{cor:compldepth}
	Let $I$ be a monomial ideal of $E$ with linear quotients. Then,
	\[\cx_E(E/I)\ =\ \max\big\{|\set(u)|:u\in G(I)\big\}.\]
	Moreover, if $K$ is an infinite field, then
	\[\depth_E(E/I)\ =\ \min\big\{n-|\set(u)|:u\in G(I)\big\}.\]
\end{Cor}
\begin{proof}
	We have $\textup{cx}_E(E/I)=\textup{cx}_E(I)$. By Corollary \ref{Cor:formulaBetti} we have
	\begin{align*}
	\beta_i(I)\ &=\ \sum_{u\in G(I)}\binom{i+\vert \set(u)\vert-1}{\vert \set(u)\vert -1}\\
	&=\ \sum_{k=1}^n\Big[\sum_{u\in G(I):|\set(u)|=k}\binom{i+k-1}{k -1}\Big]\\
	&=\ \sum_{k=1}^n\big|\big\{u\in G(I):|\set(u)|=k\big\}\big|\binom{i+k-1}{i}.
	\end{align*}
	The $k$th binomial coefficient in the previous sum is a polynomial in $i$ of degree $k-1$ and the number $\max\big\{|\set(u)|:u\in G(I)\big\}$ is the maximal $k$ which appears in the sum. It follows that $\cx_E(E/I)=\max\big\{|\set(u)|:u\in G(I)\big\}$. If $K$ is an infinite field, then by (\ref{eq:AAH}), $\depth_E(E/I)=n-\cx_E(E/I)$ and the assertion follows.
\end{proof}

\begin{Rem}\em
	If $I$ is a stable ideal of $E$, then the formula for the graded Betti numbers stated in Corollary \ref{Cor:formulaBetti} becomes formula (\ref{eq:Bettistable}), whereas the formula for the complexity in Corollary \ref{cor:compldepth} becomes the formula in \cite[Lemma 3.1.4]{GK}. In fact, $I$ has linear quotients with respect to the reverse lexicographic order and, moreover, in such a case $\vert \set(u)\vert = \m(u)$, for all $u\in G(I)$ (Remark \ref{ex:stable}). 
\end{Rem}

%%%%%%%%%%%%%%%%%%%%%%%%%%%%%%%%%%%%%%%%%
\section{Regular decomposition functions}\label{sec4}
%%%%%%%%%%%%%%%%%%%%%%%%%%%%%%%%%%%%%%%%%
In this section we describe an explicit resolution of all ideals in $E$ with linear quotients which satisfy an extra condition. For our aim we introduce the notion of regular decomposition function.

Let $I$ be a monomial ideal with linear quotients with respect to the sequence of generators $u_1, \ldots, u_r$, and set as before $I_j = (u_1, \ldots, u_j)$ for $j = 1, \ldots, r$. Let $M(I)$ be the set of all monomials in $I$. Define the map $g:M(I)\rightarrow G(I)$ as follows: $g(u) = u_j$, if $j$ is the smallest number such that $u\in I_j$.

The next statement holds.
\begin{Lem}\label{lem:JT} 
	\begin{enumerate}
		\item[\textup{(a)}] For all $u\in M(I)$ one has $u = g(u)c(u)$ with some complementary factor $c(u)$, and $\set(g(u))\cap \supp(c(u)) =\emptyset$.
		\item[\textup{(b)}] Let $u\in M(I)$, $u = vw$ with $v\in G(I)$ and $\set(v)\cap \supp(w) = \emptyset$. Then $v = g(u)$.
		\item[\textup{(c)}] Let $u, v \in M(I)$ such that $\supp(u)\cap \supp(v)= \emptyset$. Then $g(uv) = g(u)$ if and only if $\set(g(u))\cap \supp(v) = \emptyset$.
	\end{enumerate}
\end{Lem}
\begin{proof}
	The proof of Herzog and Takayama for ideals over polynomial ring \cite[Lemma 1.8]{JT} carries over.
\end{proof}

One can observe that any function $M(I)\rightarrow G(I)$ satisfying Lemma \ref{lem:JT}(a) is uniquely
determined because of Lemma \ref{lem:JT}(b). We call $g$ the \emph{decomposition function} of $I$.

\begin{Def}\em\label{def:regfunc}
	The decomposition function $g:M(I)\rightarrow G(I)$ is \emph{regular}, if\linebreak $\set(g(e_su))\subseteq \set(u)$ for all $s\in \set(u)$ such that $s\notin \supp(u)$ and $u\in G(I)$.
\end{Def}

If $I$ is a stable ideal, then its decomposition function is regular with respect to the reverse lexicographic order. In general, the decomposition function for an ideal with linear quotients is not regular as next example shows. 

\begin{Expl}\em
	Let $E=K\langle e_1, \dots, e_4\rangle$ and $I = (e_2e_4, e_1e_2, e_1e_3)$. One can verify that $I$ has linear quotients. In particular, $0:(e_2e_4)= (e_2, e_4)$, $(e_2e_4):(e_1e_2)= (e_1, e_2, e_4)$ and $(e_2e_4, e_1e_2):(e_1e_3)= (e_1, e_2, e_3)$. On the other hand, $g(e_2(e_1e_3)) = e_1e_2$ but $\set(e_1e_2) = \{1, 2, 4\}\nsubseteq \set(e_1e_3)=\{1, 2, 3\}$ and $g$ is not regular.\\
	
	\noindent
	Let us consider a new order $e_1e_2, e_2e_4, e_1e_3$ of the generators of $I$.  In such a case $0:(e_1e_2)= (e_1, e_2)$, $(e_1e_2):(e_2e_4)= (e_1, e_2, e_4)$ and $(e_1e_2, e_2e_4):(e_1e_3)= (e_1, e_2, e_3)$. Thus $I$ has linear quotients also with respect the new order of the generators. On the other hand, we can note that $\set(e_1e_2)=\{1, 2\} \subset \set(u)$, for all $u\in G(I)\setminus \{e_1e_2\}$. As a consequence, since  
	$g(e_su)=e_1e_2$, for all $s\in \set(u)\setminus \supp(u)$ and $u\in G(I)\setminus \{e_1e_2\}$, it follows that $g$ is regular. Indeed, we have that $\set(g(e_su))=\{1, 2\} \subset \set(u)$, for all $s\in \set(u)\setminus \supp(u)$ and $u\in G(I)$.
\end{Expl}

The next statement can be proved using the same techniques as the corresponding result in the polynomial case \cite[Lemma 1.11]{JTak}.
\begin{Lem}\label{Lemma:gRegularExchange}
	If the decomposition function $g:M(I)\rightarrow G(I)$ is regular, then
	\[
	g(e_s(g(e_tu)))=g(e_t(g(e_su))),
	\]
	for all $u\in M(I)$ and all $s, t\in \set(u)$ such that $s, t\notin \supp(u)$.
\end{Lem}

Let $\varepsilon_1, \ldots, \varepsilon_n$ denote the canonical basis of $\NN^n$. From Subsection \ref{sub:Cartan}, one knows that  the differential of the Cartan complex $C(e_1, \ldots, e_n; E/I)$ is given by the formula:
\[\partial(cx^{(a)}) = \sum_{i=1}^nce_i x_1^{(a_1)}\cdots x_i^{(a_i-1)}\cdots x_n^{(a_n)} = \sum_{j\in \supp(a)}ce_jx^{(a-\varepsilon_j)}, \quad c\in E/I,\quad a\in\NN^n.\]

For convenience, we extend the definition introduced in Theorem \ref{prop:iterated} setting $f(a;u)=0$ if $\supp(a) \nsubseteq \set(u)$, for $u\in M(I)$.\\

Suppose $u$ and $v$ are monomials of $E$ with $\supp(v)\subseteq\supp(u)$.
Then, there exists a unique monomial $u'\in E$ such that $u'v=u$. We set 
\[
u'=\dfrac{u}{v}.
\]
With this convention, the following identities hold:
\[
u\frac{w}{v}=\frac{uw}{v}\ \ \ \ \text{and}\ \ \ \ \frac{u}{v}\frac{v}{z}=\frac{u}{z}.
\]

In the following theorem, for our convenience, we use a modified version of the Cartan complex $C_{\cdot}(v;E)$, where $v$ is a sequence of linear forms of $E$. Let $u$ be a monomial of $E$.
We substitute the differential $\partial$ with $(-1)^{\deg(u)}\partial$ and denote the new differential again by $\partial$. We denote this modified Cartan complex by $(-1)^{\deg(u)}C_{\cdot}(v;E)$. It is clear that this complex provides again a minimal free resolution of $E/(v)$.

The next theorem is a direct analog of \cite[Theorem 1.12]{JTak}.

\begin{Thm}\label{thm:main1}
	Let $I$ be a monomial ideal of $E$ with linear quotients, and $F$ the graded minimal free resolution of $E/I$. Suppose that the decomposition function $g:M(I)\rightarrow G(I)$ is regular. Then the maps in the resolution $F$ are given by
	$d(f(0;u))=(-1)^{\deg(u)}u$, for $u\in G(I)$, and for $a\in\NN^n$, $a\ne0$, $u\in G(I)$ and $\supp(a)\subset\set(u)$,
	\begin{align*}
	d(f(a;u))\ &=\ -\sum_{t\in\supp(a)}(-1)^{\deg(u)}e_tf(a-\varepsilon_t;u)\\
	&+\sum_{t\in\supp(a)\setminus\supp(u)}(-1)^{\deg(g(e_tu))}\frac{e_tu}{g(e_tu)}f(a-\varepsilon_t;g(e_tu)).
	\end{align*}
\end{Thm}
\begin{proof}
	Let $I$ have linear quotients with respect to the sequence $u_1,\dots,u_r$. We proceed by induction on $r=|G(I)|$. For $r=1$, $I=(u_1)=(u)$ is a principal ideal and the Cartan complex $(-1)^{\deg(u)}C_{\cdot}(e_{k_1},\dots,e_{k_\ell};E)$ with $\set(u)=\supp(u)=\{k_1,\dots,k_\ell\}$ provides a minimal free resolution of $E/I$. In this case, the maps are as required.
	
	Now, suppose that $r>1$ and that our statement holds for $r-1$. As before, we set $I_{r-1} = (u_1, \ldots, u_{r-1})$ and $L_{r-1} = (u_1, \ldots, u_{r-1}):(u_{r})$, then we get the following exact sequences of $E$-modules
	\[0 \rightarrow E/L_{r-1} \longrightarrow E/I_{r-1} \longrightarrow E/I\rightarrow 0,\]
	where $E/L_{r-1} \longrightarrow E/I_{r-1}$ is multiplication by $u_{r}$. Let $F^{(r-1)}$ be the graded minimal free resolution of $E/I_{r-1}$, $C^{(r-1)}=(-1)^{\deg(u_r)}C.(e_{k_1}, \ldots, e_{k_\ell}; E)$ the Cartan complex for the sequence $e_{k_1}, \ldots, e_{k_\ell}$ with $\set(u_{r})=\{k_1,\dots,k_\ell\}$.
	
	Let $F$ be the mapping cone of $\psi^{(r-1)}:C^{(r-1)}\rightarrow F^{(r-1)}$ and for simplicity set $\psi^{(r-1)}=\psi$ and $u_r=u$. Since $F^{(r-1)}$ is a subcomplex of $F$, it follows from our inductive hypothesis that the formula for the map given in the statement holds for all $f(a;u_j)$ with $j<r$. Hence, it remains to check it for $f(a;u_r)=f(a;u)$.
	
	By the definition of the mapping cone of $\psi$,
	$d(f(a;u))=-\partial(f(a;u))+\psi(f(a;u))$. Thus, to prove the asserted formula it is enough to show that we can define $\psi$ as
	\[
	\psi(f(a;u))=\sum_{t\in\supp(a)\setminus\supp(u)}(-1)^{\deg(g(e_tu))}\frac{e_tu}{g(e_tu)}f(a-\varepsilon_t;g(e_tu))
	\]
	if $a\ne0$ and $\psi(f(0;u))=(-1)^{\deg(u)}u$, otherwise. That is, the following diagram is commutative
	\[
	\xymatrixcolsep{2.5pc}\xymatrix{
		C^{(r-1)}:\cdots\ar[r] &C_i \ar[d]_{\psi_{i}}\ar[r]^{\partial_i} & C_{i-1} \ar[d]^{\psi_{i-1}}\ar[r] & \cdots\\
		F^{(r-1)}:\cdots\ar[r] & F_i^{(r-1)} \ar[r]_-{d_i} & F_{i-1}^{(r-1)}\ar[r]&\cdots
	}\]

	Here, with abuse of notation, we also let $d$ denote the chain map of $F^{(r-1)}$.
	
	Let $t\in\set(u)$. Then
	\[
	(\psi\circ\partial)(f(\varepsilon_t;u))\ =\ \psi((-1)^{\deg(u)}e_tf(0;u))\ =\ (-1)^{\deg(u)}(-1)^{\deg(u)}e_tu=e_tu.
	\]
	On the other hand, $d(f(0;g(e_tu)))=(-1)^{\deg(g(e_tu))}g(e_tu)$. Hence, if $t\notin\supp(u)$, we have that
	
	\begin{align*}
	(d\circ\psi)(f(\varepsilon_t;u))\ &=\ (-1)^{\deg(g(e_tu))}\frac{e_tu}{g(e_tu)}d(f(0;g(e_tu)))\\
	&=\ \frac{e_tu}{g(e_tu)}g(e_tu)=e_tu.
	\end{align*}
	Otherwise, if $t\in\supp(u)$, then $\psi(f(\varepsilon_t;u))=0$. In this case, $(d\circ\psi)(f(\varepsilon_t;u))=0$ but also $(\psi\circ\partial)(f(\varepsilon_t;u))=e_tu=0$.
	So the first square in the diagram commutes.\medskip
	
	For the other squares, let $f(a;u)$ with $\supp(a)\subseteq\set(u)$ and $|a|>1$. Then
	
	\begin{align*}
	(\psi\circ\partial)(f(a;u)) &=\sum_{t\in\supp(a)}(-1)^{\deg(u)}e_t\psi(f(a-\varepsilon_t;u))\\
	&=\sum_{\substack{t\in\supp(a)\\t\notin\supp(u)}}\sum_{\substack{s\in\supp(a)\setminus\{t\}\\ s\notin\supp(u)}}(-1)^{\deg(u)+\deg(g(e_su))}\frac{e_te_su}{g(e_su)}f(a-\varepsilon_t-\varepsilon_s;g(e_su))\\
	&=\sum_{\substack{t\in\supp(a)\setminus\supp(u)\\s\in\supp(a)\setminus(\supp(u)\cup\{t\})\\s>t}}(-1)^{\deg(u)+\deg(g(e_su))}\frac{e_te_su}{g(e_su)}f(a-\varepsilon_t-\varepsilon_s;g(e_su))\\
	&-\sum_{\substack{t\in\supp(a)\setminus\supp(u)\\s\in\supp(a)\setminus(\supp(u)\cup\{t\})\\s<t}}(-1)^{\deg(u)+\deg(g(e_su))}\frac{e_se_tu}{g(e_su)}f(a-\varepsilon_t-\varepsilon_s;g(e_su)).
	\end{align*}
	
	Here, to get the second equality, we have noted that if $s=t$ or $t\in\supp(u)$ or $s\in\supp(u)$, the corresponding term in the sum is zero. Exchanging the role of $t$ and $s$ in the second sum, we obtain
	
	\begin{equation}\label{eq:psi-1circpartiali}
	\begin{aligned}
	(\psi\circ\partial)(f(a;u))&=\sum_{\substack{t,s\in\supp(a)\\t,s\notin\supp(u)\\s>t}}(-1)^{\deg(u)+\deg(g(e_su))}\frac{e_te_su}{g(e_su)}f(a-\varepsilon_t-\varepsilon_s;g(e_su))\\\\
	&-\sum_{\substack{t,s\in\supp(a)\\t,s\notin\supp(u)\\s>t}}(-1)^{\deg(u)+\deg(g(e_tu))}\frac{e_te_su}{g(e_tu)}f(a-\varepsilon_t-\varepsilon_s;g(e_tu)).
	\end{aligned}
	\end{equation}
	
	As for the composition counterclockwise, we have
	\begin{equation*}
	(d\circ\psi)(f(a;u))=\sum_{t\in\supp(a)\setminus\supp(u)}(-1)^{\deg(g(e_tu))} \frac{e_tu}{g(e_tu)}d(f(a-\varepsilon_t;g(e_tu))).
	\end{equation*}
	Let $t\in\supp(a)\setminus\supp(u)$ and assume for a moment that $\supp(a-\varepsilon_t)\subseteq\set(g(e_tu))$. Then,
	{\small\begin{equation}\label{eq:diff}
	\begin{aligned}
	d(f&(a-\varepsilon_t;g(e_tu)))=-\sum_{s\in\supp(a-\varepsilon_t)}(-1)^{\deg(g(e_tu))}e_sf(a-\varepsilon_t-\varepsilon_s;g(e_tu))\\
	&+\sum_{s\in\supp(a-\varepsilon_t)\setminus\supp(g(e_tu))}(-1)^{\deg(g(e_sg(e_tu)))}\frac{e_sg(e_tu)}{g(e_sg(e_tu))}f(a-\varepsilon_t-\varepsilon_s;g(e_sg(e_tu))).
	\end{aligned}
	\end{equation}}
	
	We claim that this expression is also true if $\supp(a-\varepsilon_t)\not\subseteq\set(g(e_tu))$. In this case, the symbol $f(a-\varepsilon_t;g(e_tu))$ is zero. Hence, the right-hand side of (\ref{eq:diff}) should be zero, too.
	
	Indeed, pick $s\in\supp(a-\varepsilon_t)$. If $\supp(a-\varepsilon_t-\varepsilon_s)\not\subseteq\set(g(e_tu))$, then by the regularity of $g$ we have $\set(g(e_sg(e_tu)))\subseteq\set(g(e_tu))$. Therefore $\supp(a-\varepsilon_t-\varepsilon_s)\not\subseteq\set(g(e_sg(e_tu)))$ too, making the right hand side of (\ref{eq:diff}) zero.
	
	Otherwise, $\supp(a-\varepsilon_t-\varepsilon_s)\subseteq\set(g(e_tu))$. In this case $s\notin\set(g(e_tu))$, otherwise we would have $\supp(a-\varepsilon_t)\subseteq\set(g(e_tu))$, against our assumption. Hence, by Lemma \ref{lem:JT}(c) we have that $g(e_sg(e_tu))=g(e_tu)$.  Furthermore, in such a case $s\in\supp(a-\varepsilon_t)\setminus\supp(g(e_tu))$ and so
	{\small \[
	(-1)^{\deg(g(e_sg(e_tu)))}\frac{e_sg(e_tu)}{g(e_sg(e_tu))}f(a-\varepsilon_t-\varepsilon_s;g(e_sg(e_tu)))=(-1)^{\deg(g(e_tu))}e_sf(a-\varepsilon_t-\varepsilon_s;g(e_tu)).
	\]}
	Therefore, we see that the right hand side of (\ref{eq:diff}) is zero, as all summands are either zero or cancel against each other.
	
	Now, observe that
	
	\begin{align*}
	-(-1)^{\deg(g(e_tu))}(-1)^{\deg(g(e_tu))}\frac{e_tu}{g(e_tu)}e_s&=-\frac{e_tu}{g(e_tu)}e_s\\
	&=-(-1)^{\deg(e_tu)-\deg(g(e_tu))}e_s\frac{e_tu}{g(e_tu)}\\
	&=(-1)^{\deg(u)+\deg(g(e_tu))}\frac{e_se_tu}{g(e_tu)}
	\end{align*}
	and 
	\begin{align*}
	&(-1)^{\deg(g(e_tu))+\deg(g(e_sg(e_tu)))}\frac{e_tu}{g(e_tu)}\frac{e_sg(e_tu)}{g(e_sg(e_tu))}
	\\
	=\ &(-1)^{\deg(g(e_tu))+\deg(g(e_sg(e_tu)))}(-1)^{\deg(g(e_tu))}\frac{e_tu}{g(e_tu)}\frac{g(e_tu)e_s}{g(e_sg(e_tu))}
	\\
	=\ &(-1)^{\deg(g(e_sg(e_tu)))}\frac{e_tue_s}{g(e_sg(e_tu))}\ =\ (-1)^{\deg(g(e_sg(e_tu)))+\deg(u)}\frac{e_te_su}{g(e_sg(e_tu))}.
	\end{align*}
	Consequently, by (\ref{eq:diff}) and the previous calculations, we have 
	\begin{align*}
	(d\circ&\psi)(f(a;u))
	=\sum_{\substack{t\in\supp(a)\\s\in\supp(a)\setminus(\supp(u)\cup\{t\})}}(-1)^{\deg(u)+\deg(g(e_tu))}\frac{e_se_tu}{g(e_tu)}f(a-\varepsilon_t-\varepsilon_s;g(e_tu))\\
	&+\sum_{\substack{t\in\supp(a)\\s\in\supp(a-\varepsilon_t)\setminus\supp(g(e_tu))}}\!\!\!\!\!(-1)^{\deg(u)+\deg(g(e_sg(e_tu)))}\frac{e_te_su}{g(e_sg(e_tu))}f(a-\varepsilon_t-\varepsilon_s;g(e_sg(e_tu))).
	\end{align*}
	Note that in the first sum we can write $t\in\supp(a)\setminus\supp(u)$, otherwise the corresponding term in the sum is zero. Likewise, in the second sum, we can write $t\in\supp(a)\setminus\supp(u)$ and moreover $s\in\supp(a)\setminus(\supp(u)\cup\{t\})$. Indeed, if $t\in\set(u)\setminus\supp(u)$, it is clear that $\supp(g(e_tu))\subset\supp(u)\cup\{t\}$. Therefore,
	\[
	\supp(a)\setminus(\supp(u)\cup\{t\})\subseteq\supp(a-\varepsilon_t)\setminus\supp(g(e_tu)).
	\]
	If $s$ belongs to the second set and does not belong to the first one, then $s\in\supp(u)$ and the corresponding summand is zero. Hence, we can rewrite the previous equation as
	\begin{align*}
	(d\circ\psi)(&f(a;u))
	=\sum_{\substack{t\in\supp(a)\setminus\supp(u)\\s\in\supp(a)\setminus(\supp(u)\cup\{t\})}}(-1)^{\deg(u)+\deg(g(e_tu))}\frac{e_se_tu}{g(e_tu)}f(a-\varepsilon_t-\varepsilon_s;g(e_tu))\\
	&+\sum_{\substack{t\in\supp(a)\setminus\supp(u)\\s\in\supp(a)\setminus(\supp(u)\cup\{t\})}}\!\!\!\!\!(-1)^{\deg(u)+\deg(g(e_sg(e_tu)))}\frac{e_te_su}{g(e_sg(e_tu))}f(a-\varepsilon_t-\varepsilon_s;g(e_sg(e_tu))).
	\end{align*}
	Consequently, we get
	\begin{align*}
	(d\circ\psi)(f(a;u))
	&=\sum_{\substack{t,s\in\supp(a)\\t,s\notin\supp(u)\\s<t}}\!\!\!(-1)^{\deg(u)+\deg(g(e_tu))}\frac{e_se_tu}{g(e_tu)}f(a-\varepsilon_t-\varepsilon_s;g(e_tu))\\
	&-\sum_{\substack{t,s\in\supp(a)\\t,s\notin\supp(u)\\s>t}}\!\!\!(-1)^{\deg(u)+\deg(g(e_tu))}\frac{e_te_su}{g(e_tu)}f(a-\varepsilon_t-\varepsilon_s;g(e_tu))\\
	&+\sum_{\substack{t,s\in\supp(a)\\t,s\notin\supp(u)\\s>t}}\!\!\!(-1)^{\deg(u)+\deg(g(e_sg(e_tu)))}\frac{e_te_su}{g(e_sg(e_tu))}f(a-\varepsilon_t-\varepsilon_s;g(e_sg(e_tu)))\\
	&-\sum_{\substack{t,s\in\supp(a)\\t,s\notin\supp(u)\\s<t}}\!\!\!(-1)^{\deg(u)+\deg(g(e_sg(e_tu)))}\frac{e_se_tu}{g(e_sg(e_tu))}f(a-\varepsilon_t-\varepsilon_s;g(e_sg(e_tu))).
	\end{align*}
	By Lemma \ref{Lemma:gRegularExchange}, the last two sums cancel out by exchanging the roles of $s$ and $t$ in one of them. Moreover, we get an expression identical to (\ref{eq:psi-1circpartiali}) by exchanging the roles of $s$ and $t$ in the first sum. Thus, we see that $(d\circ\psi)(f(a;u))=(\psi\circ\partial)(f(a;u))$. The conclusion follows. 
\end{proof}

%%%%%%%%%%%%%%%%%%%%%%%%
\section{An application}\label{sec5}
%%%%%%%%%%%%%%%%%%%%%%%%
In this section we introduce the notion of strongly stable vector-spread ideal in the exterior algebra $E = K\langle e_1, \ldots, e_n\rangle$. This notion has been recently introduced in \cite{F1} in the polynomial ring as a generalization of the concept of $t$-spread strongly stable ideal given by Ene, Herzog and Qureshi in \cite{EHQ}.

\begin{Def}\em \label{def:spread}
	Given $n\ge1,{\bf t}=(t_1,\dots,t_{d-1})\in\NN^{d-1},d\ge2$ and $e_\mu=e_{i_1} \cdots e_{i_\ell}\neq 1$ a monomial of $E$, with $1\le i_1< i_2< \cdots < i_\ell\le  n$, $\ell\le d$ we say that $e_\mu$ is \emph{${\bf t}$-spread} if $i_{j+1}-i_j\ge t_j$, for all $j=1,\dots,\ell-1$. We say that a monomial ideal $I$ of $E$ is \emph{${\bf t}$-spread} if it is generated by ${\bf t}$-spread monomials.
\end{Def}

For instance, the monomial ideal $I=(e_1e_3, e_2e_6, e_2e_4e_8)$ of the exterior algebra $E = K\langle e_1, \ldots, e_8\rangle$ is a $(2,2)$-spread monomial ideal, but not a $(3,2)$-spread monomial ideal. \\

Note that any monomial ideal of $E$ is ${\bf 1}$-spread, where ${\bf 1}=(1,\dots,1)\in \NN^n$. The set of all monomials of $E$ will be denoted by $\Mon(E)$, whereas the set of all ${\bf t}$-spread monomials of $E$ will be denote by $\Mon_{\bf t}(E)$. For ${\bf t}={\bf 1}$, $\Mon_{\bf 1}(E)=\Mon(E)$. 
One can quickly observe that if $\ell$ is a positive integer, then there does exist a ${\bf t}$-spread monomial of degree $\ell$ in $E$ if and only if $n\ge 1+\sum_{j=1}^{\ell-1}t_j$. 

\begin{Def}\em\label{def:stronspread}
	A ${\bf t}$-spread monomial ideal $I$ is called \emph{${\bf t}$-spread strongly stable} if for all ${\bf t}$-spread monomial $e_\mu\in I$, all $j\in \supp(e_\mu)$ and all $i< j$, such that $e_ie_{\mu\setminus \{j\}}$ is ${\bf t}$-spread, it follows that $e_ie_{\mu\setminus \{j\}}\in I$.
\end{Def}

For instance, the ideal $I= (e_1e_3, e_1e_4, e_2e_4e_6)$ is a $(2,2)$-spread strongly stable monomial ideal of $E = K\langle e_1, \ldots, e_6\rangle$.

Note that a ${\bf 1}$-spread strongly stable ideal is the ordinary strongly stable ideal introduced in \cite{AHH}.\\

The defining property of a ${\bf t}$-spread strongly stable ideal needs to be checked only for the set of ${\bf t}$-spread monomial generators.
\begin{Prop}\label{rem:stable}  
	Let $I$ be a ${\bf t}$-spread ideal and suppose that for all $e_\mu \in G(I)$, and for all integers $1\le i < j \le n$ such that $j\in \mu$ and  $e_ie_{\mu \setminus \{j\}}$ is ${\bf t}$-spread, one has $e_ie_{\mu \setminus \{j\}} \in I$. Then $I$ is ${\bf t}$-spread strongly stable. 
\end{Prop}
\begin{proof}
	Let $e_\nu \in I$ be a ${\bf t}$-spread monomial and $1 \le i < j \le n$ integers such that $j\in \nu$ and $e_ie_{\nu \setminus \{j\}}$ is ${\bf t}$-spread. There exists $e_\mu \in G(I)$ and a monomial $e_\rho \in E$ such that $e_\nu = e_\mu e_\rho$ in $E$. We distinguish two cases: $j \in \mu$, $j \in \rho$.
	
	If $j \in \mu$, then $e_ie_{\mu \setminus \{j\}}  \in I$ by assumption, and so $e_ie_{\nu \setminus \{j\}} = e_ie_{\mu \setminus \{j\}} e_\rho \in I$.
	
	If $j\in \rho$, then $e_ie_{\nu \setminus \{j\}}  = e_ie_\mu e_{\rho \setminus \{j\}} \in I$.
\end{proof}

In what follows, let ${\bf t}=(t_1,\dots,t_{d-1})\in\NN^{d-1}$, $d\ge2$.

The proof of the next statement is verbatim the same as the corresponding result in the polynomial ring, see \cite[Theorem 2.2]{CF} and \cite[Proposition 5.1]{F1}. 
Thus we omit it.

\begin{Thm}\label{thm:collect}
	Let $I$ be a ${\bf t}$-spread strongly stable ideal. Then $I$ has linear quotients with $G(I)$ ordered with respect to the lexicographic order. In particular, $I$ is componentwise linear.
\end{Thm}

Let $S=K[x_1, \ldots, x_n]$ be the standard graded polynomial ring on $n$ variables.  By abuse of notation, and in order to simplify the notation, if $u$ is a monomial in $E$, we denote again by $u$ the corresponding squarefree monomial in $S$. For instance, if $u=e_1e_2e_3\in \Mon(E)$ then we denote by $u$ the squarefree monomial $x_1x_2x_3$ of $S$, too.  Moreover, if $I$ is a monomial ideal in $E$, we will denote again by $I$ the corresponding squarefree ideal in $S$.

Let $I$ be a ${\bf t}$-spread strongly stable ideal of $E$ with $G(I) = \{u_1,\dots,u_r\}$ ordered with respect to the pure lexicographic order \cite{JT}. Using the notation discussed above, $I$ can be considered as a ${\bf t}$-spread strongly stable ideal of $S$ with $G(I) = \{u_1,\dots,u_r\}$ ordered with respect to the pure lexicographic order. Hence, we set
\[
\set_S(u_j)=\{i:x_i\in(u_1,\dots,u_{j-1}):_S(u_j)\}, \quad \set_E(u_j)=\{i:e_i\in(u_1,\dots,u_{j-1}):_E(u_j)\},
\]
for $j=1, \ldots, r$. It is clear that 
\[\set_E(u_j)=\set_S(u_j)\cup\supp(u_j), \quad j=1, \ldots, r.\] 
Now, by \cite[Corollary 2.3]{CF}, if $u_k=e_{j_1}\cdots e_{j_\ell}$, then
\[
\set_S(u_k)=[\m(u_k)-1]\setminus\bigcup\limits_{h=1}^{\ell-1}\ \big[j_h,j_h+(t_h-1)\big],
\]
where $[a,b]=\{c\in\NN:a\le c\le b\}$ and $a,b\in\NN$.

Therefore,
\begin{equation}\label{eq:compset}
\set_E(u_k)=\set_S(u_k)\cup\supp(u_k)=[\m(u_k)]\setminus\bigcup\limits_{h=1}^{\ell-1}\ \big[j_h+1,j_h+(t_h-1)\big]
\end{equation}
and consequentially $|\set_E(u_k)|=\m(u_k)-\sum_{h=1}^{\ell-1}(t_h-1)$.

Theorem \ref{thm:collect} yields the following corollary which allow us to state a formula for computing the graded Betti numbers of a ${\bf t}$-spread strongly stable ideal.

\begin{Cor}\label{cor:formula}
	Let $I$ be a ${\bf t}$-spread strongly stable ideal. Then
	\[
	\beta_{i, i+j}(I)=\sum_{u\in G(I)_j}\binom{\m(u)-\sum_{h=1}^{j-1}(t_h-1)+i-1}{i}.
	\]
\end{Cor}
\begin{proof}
	From Theorem \ref{thm:collect}, $I$ has linear quotients. Hence, the assertion follows from Corollary \ref{Cor:formulaBetti} and (\ref{eq:compset}).
\end{proof}

\begin{Rem}\em
	For ${\bf t}={\bf 1}$, the formula in Corollary \ref{cor:formula} becomes the well--known formula (\ref{eq:Bettistable}).
\end{Rem}

% ------------------------------------------------------------------------
\subsection*{Acknowledgment}
We thank the anonymous referees for his/her careful reading and helpful suggestions.


\begin{thebibliography}{1}
\bibitem{AC} Amata L., Crupi M. (2018). ExteriorIdeals: A package for computing monomial ideals in an exterior algebra. \textit{JSAG}. {\bf 8}: 71--79. DOI: http://dx.doi.org/10.2140/jsag.2018.8.71

\bibitem{AAH1} Aramova, A.,  Avramov, L.L.,  Herzog, J. (2000). Resolutions of monomial ideals and cohomology over exterior algebras. \textit{Trans. Am. Math. Soc.} {\bf 352}(2): 579--594. DOI 10.1090/s0002-9947-99-02298-9

\bibitem{AHH} Aramova, A., Herzog, J., Hibi, T. (1997). Gotzmann Theorems for Exterior algebra and combinatorics. \textit{J. Algebra}. {\bf 191}: 174--211. DOI: 10.1006/jabr.1996.6903

\bibitem{BW} Bj\"{o}rner, A.,  Wachs, M.L. (1996). Shellable nonpure complexes and posets, I, \textit{Trans. Amer. Math. Soc.} {\bf 348}(4): 1299--1327. DOI: 10.1090/s0002-9947-96-01534-6

\bibitem{CF} Crupi M., Ficarra A., (2023). A note on minimal resolutions of vector-spread Borel ideals, \textit{An. St. Univ. Ovidius Constantia}, Seria Matematica {\bf 31}(2): 71--84. DOI: 10.2478/auom-2023-0020

\bibitem{DM} Dochtermann A., Mohammadi F., (2014). Cellular resolutions from mapping cones, \textit{J. Combin.Theory Ser.A 1} {\bf 28}: 180--206. DOI: 10.1016/j.jcta.2014.08.007

\bibitem{DE} Eisenbud D. (1995). \textit{Commutative algebra: with a view toward algebraic geometry}, \textit{Graduate Texts in Mathematics}, Vol. 150, New York, USA: Springer New York. https://doi.org/10.1007/978-1-4612-5350-1

\bibitem{EK} Eliahou S., Kervaire M., (1990). Minimal resolutions of some monomial ideals, \textit{J. Algebra} {\bf 129}: 1--25. DOI: 10.1016/0021-8693(90)90237-I

\bibitem{EHQ} Ene V., Herzog J., Qureshi A.A., (2019). t-spread strongly stable monomial ideals, \textit{Comm. Algebra} {\bf 47}(12): 5303--5316. DOI: 10.1080/00927872.2019.1617876

\bibitem{EC} Evans G., Charalambous H., (1995). Resolutions obtained by iterated mapping cones, \textit{J. Algebra} {\bf 176}: 75--754. DOI: 10.1006/jabr.1995.1270

\bibitem{FH} Ferraro L., Hardesty A., (2023). The Eliahou--Kervaire resolution over a skew polynomial ring, \textit{Comm. Algebra}. DOI: 10.1080/00927872.2023.2269560

\bibitem{F1} Ficarra A., (2023). Vector-spread monomial ideals and Eliahou--Kervaire type resolution, \textit{J. Algebra} {\bf 615}: 170--204. DOI: 10.1016/j.jalgebra.2022.10.017

\bibitem{GDS} Grayson D.R., Stillman M.E., \textit{Macaulay2, a software system for research in algebraic geometry}, available at \url{http://www.math.uiuc.edu/Macaulay2}. 

\bibitem{HH1999} Herzog J., Hibi T., (1999). Componentwise linear ideals, \textit{Nagoya Math. J.} {\bf 153}: 141--153. DOI: https://doi.org/10.1017/S0027763000006930

\bibitem{JT} Herzog J., Hibi T., (2011). \textit{Monomial ideals}, \textit{Graduate texts in Mathematics}, Vol. 260, London, UK: Springer--Verlag. https://doi.org/10.1007/978-0-85729-106-6 

\bibitem{JTak} Herzog J., Takayama Y., (2002). Resolutions by mapping cones, \textit{Homology, Homotopy and Applications}, {\bf 4}(2): 277--294. DOI: 10.4310/hha.2002.v4.n2.a13

\bibitem{GK} K\"{a}mpf G., (2010). Module theory over exterior algebra with applications to combinatorics, Dissertation zur Erlangung des Doktorgrades, Fachbereich Mathematik/Informatik, Universit\"{a}t Osnabr\"{u}ck.

\bibitem{EEE22} Mastroeni M., McCullough J., Osborne A., Rice J., Willis C., (2022). Depth and Singular Varieties of Exterior Edge Ideals, arXiv preprint arXiv:2208.03366

\bibitem{HM} Matsumura H. (1987). \textit{Commutative ring theory}. (Reid, M., trans.), \textit{Cambridge Studies in Advanced Mathematics}, Vol. 8, Cambridge, UK: Cambridge University Press. https://doi.org/10.1017/CBO9781139171762

\bibitem{IP} Peeva I., (2011). \textit{Graded Syzygies}, \textit{Algebra and Applications}, Vol. 14, London, UK: Springer--Verlag. https://doi.org/10.1007/978-0-85729-177-6 

\bibitem{Tthesis} Thieu D. P., (2013). On graded ideals over the exterior algebra with applications to hyperplane arrangements, Dissertation zur Erlangung des Doktorgrades, Fachbereich Mathematik/Informatik, Universit\"{a}t Osnabr\"{u}ck.

\bibitem{KV} VandeBogert K., (2021). Iterated mapping cones for strongly Koszul algebras, arXiv preprint arXiv:2104.00037

\bibitem{CW} Weibel C.A., (1994). \textit{An introduction to Homological Algebra}, \textit{Cambridge Studies in Advanced Mathematics}, Cambridge: Cambridge Univ. Press. https://doi.org/10.1017/CBO9781139644136
\end{thebibliography}
\end{document}